\documentclass[12pt, a4paper, reqno, oneside]{article}
\usepackage{enumerate}

\usepackage{xcolor}
\usepackage{tikz, pgfplots}
\usepackage{calc}

\usepackage{amsmath, amssymb, amsthm}

\usepackage{marginnote}
\usepackage{ragged2e}

\usepackage[colorlinks=true, pdfstartview=FitV, linkcolor=blue, citecolor=blue, urlcolor=blue,pagebackref=false]{hyperref}

\hypersetup{ colorlinks, citecolor=black, filecolor=black, linkcolor=blue, urlcolor=black }

\usepackage{fullpage}

\usepackage{type1cm} 

\newtheorem{theorem}{Theorem}[section]
\newtheorem{corollary}[theorem]{Corollary}

\newtheorem{lemma}[theorem]{Lemma}

\newtheorem{question}[theorem]{Question}

\numberwithin{equation}{section}

\theoremstyle{definition}
\newtheorem{definition}[theorem]{Definition}
\theoremstyle{remark}
\newtheorem{remark}[theorem]{Remark}

\renewcommand{\leq}{\leqslant}  
\renewcommand{\geq}{\geqslant}

\newcommand{\mcup}{\textstyle \bigcup\limits}
\newcommand{\mcap}{\textstyle \bigcap\limits}

\DeclareMathOperator\diam{diam}
\DeclareMathOperator\Var{Var}
\DeclareMathOperator\Vol{Vol}

\usepackage{titlesec}
\titleformat{\section}{\Large\bfseries}{\thesection}{1em}{}
\titleformat{\subsection}{\bfseries}{\thesubsection}{1em}{}

\usepackage{courier}

\usepackage{array}
\newcolumntype{e}{>{\displaystyle}r @{\,} >{\displaystyle}c @{\,} >{\displaystyle}l}

  \newcounter{constant}
  \newcommand{\newconst}[1]{\refstepcounter{constant}\label{#1}}
  \newcommand{\useconst}[1]{c_{\textnormal{\tiny \ref{#1}}}}
\def\clap#1{\hbox to 0pt{\hss#1\hss}}

\usepackage{calc}

\def\arraypar#1{\parbox[c]{\textwidth - 2cm}{\centering #1}}

\usepackage{environ}
\NewEnviron{display}{
  \begin{equation}\begin{array}{c}
    \arraypar{\BODY}
  \end{array}\end{equation}
}

\begin{document}

\title{Percolation and local isoperimetric inequalities} 

\author{Augusto Teixeira\footnote{IMPA, Estrada Dona Castorina 110, 22460-320, Rio de Janeiro RJ, Brazil\newline
e-mail: {\itshape \texttt{augusto@impa.br}}}}

\date{\today}
\maketitle

\begin{abstract}
  In this paper we establish some relations between percolation on a given graph $G$ and its geometry.
  Our main result shows that, if $G$ has polynomial growth and satisfies what we call the \emph{local isoperimetric inequality} of dimension $d > 1$, then $p_c(G) < 1$.
  This gives a partial answer to a question of Benjamini and Schramm from \cite{BS96}.
  As a consequence of this result and \cite{MR1765172} we derive, under the additional condition of bounded degree, that these graphs also undergo a non-trivial phase transition for the Ising-Model, the Widom-Rowlinson model and the beach model.
  Our techniques are also applied to dependent percolation processes with long range correlations.
  We provide results on the uniqueness of the infinite percolation cluster and quantitative estimates on the size of finite components.
  Finally we leave some remarks and questions that arise naturally from this work.

  \vspace{0.5cm}\noindent {\it Math. Subject Classification:} 60K35, 82B43, 05C10.
\end{abstract}

\section{Introduction}

The mathematical interest in percolation dates back to the works of Broadbent and Hammersley, \cite{PSP:2048852}.
This simple mathematical model they have introduced on $\mathbb{Z}^d$ (more generally on crystals) has motivated intense research both in the physics and mathematics literature, giving rise to beautiful theories and challenges.
Excellent introductions to the mathematical aspects of the model in $\mathbb{Z}^d$ can be found in \cite{Gri99} and \cite{bollobas2006percolation}.

Since its introduction, this model has found different applications besides the physical process of percolation, ranging from network analysis (electric and social networks, internet and the world wide web), disease and rumor propagation, among others.
However, when one focuses on these other applications, the original graph $\mathbb{Z}^d$ may no longer be the most natural setting to define the model.

With the seminal works \cite{BS96} and \cite{LP11}, the study of percolation on more general graphs received a much wider attention.
In particular, the authors of \cite{BS96} layed down several open questions that motivated and guided the continuation of this study.

In this paper we deal with a fundamental question on this subject, concerning the existence of a non-trivial phase transition for site percolation on $G$. Define
\begin{equation}
  \label{e:p_c}
  p_c = \sup \big\{ p \in [0,1]; \mathbb{P}[\text{there exists an infinite open cluster}] = 0 \big\}.
\end{equation}

Quoting Benjamini and Schramm,
\begin{quote}
``The first step in a study of percolation on other graphs $(\cdots)$ will be to prove that the critical probability on these graphs is smaller than one.''
\end{quote}

In \cite{BS96}, the authors predicted that isoperimetric inequalities should play an important role in this task.
We say that a graph $G = (V, \mathcal{E})$ has isoperimetric dimension at least $d$ if
\begin{equation}
  \label{e:standard_iso}
  \begin{array}{c}
    \text{for any finite $A \subseteq V$, we have } |\partial A| \geq c|A|^{\tfrac{d-1}{d}},
  \end{array}
\end{equation}
for some constant $c > 0$ independent of $A$.
An important example of graph satisfying this is the Euclidean lattice $\mathbb{Z}^d$, see Remark~\ref{r:isoperimetric}.
In \cite{BS96}, the authors posed the following
\begin{question}
  \label{q:BS}
  If $G$ has isoperimetric dimension at least $d > 1$, then $p_c(G) < 1$?
\end{question}
Benjamini and Schramm have solved the case $d = \infty$, that is, they have shown that $p_c(G) < 1$ if $|\partial A| \geq c|A|$ holds for all $A \subseteq V$, see Theorem~2 of \cite{BS96}.
Kozma, in \cite{zbMATH05229215}, answered this question affirmatively for planar graphs of polynomial growth.
In Subsection~\ref{ss:previous}, we are going to review this and some other results in this direction.

Here we deal solely with graphs having polynomial growth. Let us introduce the
\begin{definition}
  \label{d:volume_bound}
  We say that a graph $G$ satisfies the \emph{volume upper bound} $\mathcal{V}_u(d_u, c_u)$ if
  \begin{equation}
    \label{e:volume_bound}
    |B(x,r)| \leq c_u r^{d_u}, \text{ for all $x \in V$ and $r \geq 1$}.
  \end{equation}
\end{definition}

Moreover, for our results we require a slightly modified version of the isoperimetric inequality \eqref{e:standard_iso}, resambling the definition that appears in the Appendix of \cite{kanai1985}.
Given sets $A \subseteq B \subseteq V$, we define $\partial_B A$ to be the edge boundary of $A$ when looked as a subset of the graph induced by $B$ in $G$, see the definition in Section~\ref{s:notation}.

\begin{definition}
  \label{d:isoperimetric}
  We say that a graph $G$ satisfies the \emph{local isoperimetric inequality} $\mathcal{L}(d_i, c_i)$ if for every $x \in V$, $r \geq 1$ and $A \subseteq B(x, r) =: B$ such that $|A| \leq |B(x,r)|/2$ we have
  \begin{equation}
    \label{e:isoperimetric}
    |\partial_B A| \geq c_i |A|^{\frac{d_i-1}{d_i}} .
  \end{equation}
\end{definition}

Note the similarity between the above definition and \eqref{e:standard_iso}.
Again, it is easy to see that $\mathbb{Z}^d$ satisfies the above, see Remark~\ref{r:isoperimetric}.
We call the above \emph{local} isoperimetric inequality because we are bounding from below $|\partial_B A|$ instead of $|\partial A|$.
This distinction is important through our arguments and is further discussed in Remark~\ref{r:isoperimetric}.
We can now state
\begin{theorem}
  \label{t:bernoulli}
  If $G$ satisfies the local isoperimetric inequality $\mathcal{L}(d_i,c_i)$, together with the volume bound $\mathcal{V}_u(d_u, c_u)$, for arbitrary $d_i, d_u > 1$, then $p_c < 1$.
  More precisely, there is a $p_* = p_*(d_i, c_i, d_u, c_u, d_l, c_l) < 1$ such that, for every $p > p_*$, one has $\mathbb{P}$-a.s. a unique open infinite connected component.
  Moreover, for every $\chi > 0$ and $p$ large enough,
  \begin{equation}
    \label{e:second_cluster}
    \lim_{V \to \infty} V^\chi \sup_{x \in V} \mathbb{P}[V < |\mathcal{C}_x| < \infty] = 0.
  \end{equation}
  Where $\mathcal{C}_x$ stands for the connected component containing $x$.
\end{theorem}

It is interesting to notice that Theorem~\ref{t:bernoulli} also provides the uniqueness of the infinite connected component $\mathcal{C}_\infty$ for large parameters $p$.
In light of the work of Burton and Keane \cite{BK89}, this uniqueness statement may sound redundant, as Theorem~\ref{t:bernoulli} supposes that $G$ has polynomial growth.
Observe however that we are not requiring $G$ to be transitive as it is the case in \cite{BK89}. See also \cite{HJ06} for a discussion on this subject.

A simple modification of our arguments, could improve the polynomial factor $V^\chi$ in \eqref{e:second_cluster} by $V^{\log(V)}$ for instance.
It would be interesting to attempt to extract a better decay for the tail of the finite clusters, see more on this in Remark~\ref{r:questions}.

Other advantages of our techniques is that they do not look much into the details of the system.
In particular, we do not require the graph $G$ to be transitive, or have bounded degree.
Perhaps more surprisingly, our proof can be extended to deal with highly dependent percolation models, as we state in the next section.

\subsection{Statistical mechanics models and dependent percolation}

We now present a consequence of Theorem~\ref{t:bernoulli} for statistical mechanics models on graphs.
We consider the Ising model, the Widom-Rowlinson and the beach model on $G$.
\begin{definition}
  We say that a model undergoes a non-trivial phase transition on $G$ if for some parameter it presents more than one Gibbs measure.
  See Section~2.3 of \cite{MR1765172} for more details.
\end{definition}

Theorem~\ref{t:bernoulli} has the following consequence for such models.
\begin{corollary}
  \label{c:ising}
  If $G$ has bounded degree and satisfies $\mathcal{L}(d_i,c_i)$ and $\mathcal{V}_u(d_u, c_u)$, for arbitrary $d_i, d_u > 1$, then it presents a non-trivial phase transition for the Ising, Widom-Rowlinson and the beach models.
  We can drop the condition of bounded degree for the case of the Ising model.
\end{corollary}

\begin{proof}
  This result is a direct consequence of Theorem~\ref{t:bernoulli} in light of Theorems~1.1 and 1.2 of \cite{MR1765172}.
\end{proof}

Another important observation for us comes from \cite{zbMATH03172619}, where it is proved that $p_c < 1$ for site percolation implies $p_c < 1$ for bond percolation.
This is the main reason why we focus on the case of site percolation.

Besides the above described statistical mechanics models, we are able to treat dependent percolation processes with polynomial decay of correlations.
This is the content of our next result.

To state it, we need to introduce a way to quantify the dependence of a random environment.
Intuitively speaking, the condition below requires that what happens in two well separated regions of the graph should be approximately independent.

Let $\mathbb{P}$ be any probability measure on $\Omega = \{0,1\}^{V}$, endowed with the $\sigma$-algebra generated by the canonical projections $Y_x:\Omega \to \{0,1\}$, for $x \in V$, given by $Y_x(\omega) = \omega(x)$.

\begin{definition}
  \label{d:decouple}
  We say that $\mathbb{P}$ satisfies the \emph{decoupling inequality} $\mathcal{D}(\alpha, c_\alpha)$ if for any $x \in V$, $r \geq 1$ and any decreasing events $\mathcal{G}$, $\mathcal{G}'$ such that
  \begin{equation}
    \mathcal{G} \in \sigma(Y_z; z \in B(x,r)) \qquad \text{and} \qquad \mathcal{G}' \in \sigma(Y_z; z \not \in B(x,2r)),
  \end{equation}
  we have
  \begin{equation}
    \label{e:decouple}
    \mathbb{P}(\mathcal{G} \cap \mathcal{G}') \leq \mathbb{P}(\mathcal{G}) \mathbb{P}(\mathcal{G}') + c_\alpha r^{-\alpha}.
  \end{equation}
\end{definition}

In analogy with the volume upper bound $\mathcal{V}_u(d_u, c_u)$, we introduce the
\begin{definition}
  \label{d:volume_bound2}
  We say that a graph $G$ satisfies the \emph{volume lower bound} $\mathcal{V}_l(d_l, c_l)$ if
  \begin{equation}
    \label{e:volume_bound2}
    |B(x,r)| \geq c_l r^{d_l}, \text{ for any $x \in V$ and $r \geq 1$.}
  \end{equation}
  Observe that every connected infinite graph satisfies $\mathcal{V}_l(1, 1)$.
\end{definition}

The following theorem is a generalization of Theorem~\ref{t:bernoulli} for dependent $\mathbb{P}$'s.
\begin{theorem}
  \label{t:main}
  Suppose that $G$ satisfies the local isoperimetric inequality $\mathcal{L}(d_i,u_i)$ and the volume bounds $\mathcal{V}_u(d_u, c_u)$ and $\mathcal{V}_l(d_l, c_l)$, for arbitrary $d_i > 1$, $d_u, d_l \geq 1$.
  Assume moreover that the random environment given by $\mathbb{P}$ on $\{0,1\}^V$ satisfies $\mathcal{D}(\alpha, c_\alpha)$ with
  \begin{equation}
    \label{e:alpha_large}
    \alpha > \Big(1 \vee \frac{d_i (d_u - 1)}{d_i - 1} \Big) d_u - d_l.
  \end{equation}
  Then, there is $p_* = p_*(d_i, c_i, d_u, c_u, d_l, c_l, \alpha, c_\alpha) < 1$ such that, if $\, \inf_{x \in V} \mathbb{P}[\text{$x$ is open}] > p_*$, there exists $\mathbb{P}$-a.s. a unique infinite connected component $\mathcal{C}_\infty$.
  Moreover, for $\chi > 0$ and $p$ close enough to one, we have
  \begin{equation}
    \label{e:second_cluster}
    \lim_{V \to \infty} V^\chi \sup_{x \in V} \mathbb{P}[V < |\mathcal{C}_x| < \infty] = 0.
  \end{equation}
  Note that when $d_i = d_u = d_l$, the condition \eqref{e:alpha_large} on $\alpha$ reduces to $\alpha > d_i (d_i - 1)$.
\end{theorem}
If $G$ has bounded degree, we can use Lemma~\ref{l:volume_lower} to replace $d_l$ in \eqref{e:alpha_large} by $d_i$.

\begin{remark}
  Let us observe that the constant $p_*$ in Theorem~\ref{t:main} only depends on the graph $G$ and the law $\mathbb{P}$ through the constants $d_i, c_i, d_u, c_u, d_l, c_l, \alpha$ and $c_\alpha$.
  Note the similarity between this fact and the result in Theorem~1.2 of \cite{BLPS97}, dealing with non-amenable graphs.

  In Theorem~1.1 of \cite{BLPS97}, the authors prove that, for amenable graphs, one cannot drop the dependence of $p_*$ on $\alpha$ and $c_\alpha$, otherwise $p_*$ would not be uniform over the measures $\mathbb{P}$.
  An interesting example of percolation process with polynomial decay of correlations that presents no phase transition is given in \cite{TW10b}, Proposition~5.6.
\end{remark}

\subsection{Previously known results}
\label{ss:previous}

Besides the lattice $\mathbb{Z}^d$, other important examples of graphs received special attention in the literature, such as regular trees, the complete graph, fractal-type graphs, hyper-cubes and others.
Due to their symmetry, some of these examples proved to be simpler to analyze than the original setting $\mathbb{Z}^d$.

There is also a rich literature that studies percolation on graphs under various conditions.
We now mention some works, where the question of whether $p_c(G) < 1$ has been attacked.

If a graph has positive Cheeger's constant, the fact that $p_c < 1$ has been established in Theorem~2 of \cite{BS96}.
The case of Cayley graphs with exponential growth has been investigated in \cite{zbMATH01462554} and \cite{BLPS97}, see also \cite{LP11}.
Cayley graphs of finitely presented groups with one end, have been covered in Corollary~10 of \cite{zbMATH01224777}.
In \cite{MR2054174}, the authors prove that $p_c < 1$ under several conditions, the main one is called the \emph{minimal cut-set property}.
Question~\ref{q:BS} has been answered positively for the case of planar graphs of polynomial growth in \cite{zbMATH05229215}.

In the aforementioned works, several ideas and techniques have been used, such as mass transport principles, analytical tools, exploration algorithms and energy vs entropy estimations.
Roughly speaking, the approach we devise here follows an energy vs entropy strategy, but understood from a renormalization perspective, that we now briefly describe.

\subsection{Idea of the proof}

Suppose we are able to find a combinatorial structure that is necessary to prevent percolation, one could think for instance on the existence of a dual circuit surrounding the origin in the case of $\mathbb{Z}^2$.
Then, the task of showing the existence of an infinite cluster is reduced to ruling out the existence of such blocking structures.
The energy vs entropy approach consists in showing that, for $p$ close enough to one, the cost of observing a given blocking structure overwhelms their combinatorial richness.

Here we employ a similar technique, where our blocking structures are induced by the so called separation events $\mathcal{S}(x,L)$, introduced in Definition~\ref{d:SxL}.
This definition finds some inspiration in \cite{T10} and have two important features that are well suited to this work.

First of all, they are hierarchical in nature, as proved in Lemma~\ref{l:cascade}, allowing us to employ a renormalization procedure to bound their probabilities.
The second important property of the separation events is that we can choose the precise way in which they interact between scales, see Lemma~\ref{l:cascade}.
Therefore, we can adapt our arguments to the specific isoperimetric profile of $G$.
This way, our results apply for any $d_i > 1$ as stated above.

The flexibility of our techniques in dealing with dependent environments contrasts with other methods that look closely into the microscopic shape of the blocking structures, such as Peierls argument, see for instance the proof of Theorem~1.10 in \cite{Gri99} p.15.

This paper is organized as follows.
In Section~\ref{s:notation} we introduce some notation and simple graph theoretical results needed throughout the text.
The notion of separation events and the renormalization scheme that is used throughout our proofs are presented in Section~\ref{s:separation}.
In Section~\ref{s:proofs} we provide the proofs of Theorems~\ref{t:bernoulli} and \ref{t:main}.
We leave some questions and remarks in Section~\ref{s:questions}.

{\bf Acknowledgments - } We would like to thank Itai~Benjamini for the interesting remarks and for letting me know of the reference \cite{kanai1985} appearing in Remark~\ref{r:isoperimetric}.
We are also grateful to Elisabetta Candellero for her comments and reading.
This research has been supported by CNPq grants 306348/2012-8 and 478577/2012-5.

\section{Notation and preliminary results}
\label{s:notation}

In this section we establish some notation needed in the course of the article, as well as some results on graph theory.
Although some of these preliminary results are reasonably simple, we provide their proof for the sake of completeness.

Let us first comment on our use of constants.
We use $c$ for a positive and finite constant that may change from line to line.
Should a constant depend on further parameters, such as $d_i, c_i, \dots$ this dependence will be indicated like in $c(d_i, c_i)$.
More important constants are numbered as $c_0, c_1, \dots$ and refer to their first appearance in the text.

Throughout this article, we will write $G = (V,\mathcal{E})$ for an infinite connected graph with finite geometry, that is, we assume that every vertex has only a finite number of neighbors.

Given a set $A \subseteq V$, we denote its boundary by $\partial A = \big\{\{x,y\} \in \mathcal{E}; x \in A \text{ and } y \not \in A\big\}$.
For sets $A \subseteq B \subseteq V$, we introduce the edge boundary of $A$ relative to $B$ through the following $\partial_B A = \big\{ \{x,y\} \in \mathcal{E}; x \in A, y \in B \setminus A \big\}$.

We call $\sigma:\{0,1,\dots,l\} \to V$ a path if $\{\sigma_{i-1}, \sigma_i\} \in \mathcal{E}$ for every $i = 1,\dots,l$.
The integer $l$ above is called the length of $\sigma$.
Such a path is said to be open if $Y_{\sigma_i} = 1$ for every $i = 0, \dots, l$.

Given $x, x' \in V$, we write $d(x,x')$ for the smaller length among all paths starting at $x$ and ending at $x'$.
The distance between two sets $d(A,A')$ is given by the minimum distance between points $x \in A$ and $x' \in A'$ and analogously for $d(A,x)$.
For $x \in V$ and $r \in \mathbb{R}_+$ we define $B(x,r) = \{y \in V; d(x,y) \leq r\}$ and if $K \subseteq V$, we denote the $r$-neighborhood of $K$ by $B(K,r) = \{y \in V; d(K,y) \leq r\}$.

A map $\omega \in \Omega := \{0,1\}^V$ is called a site percolation configuration and we endow the set $\Omega$ with the  $\sigma$-algebra generated by the canonical projections $(Y_x)_{x \in V}$ and a probability measure $\mathbb{P}$.
In Theorem~\ref{t:bernoulli}, $\mathbb{P}$ is taken to be the product measure, under which the variables $(Y_x)_{x \in V}$ are independent with $\mathbb{P}[Y_x = 1] = p \in [0,1]$.
Given a configuration $\omega \in \Omega$ and $x \in V$, we define $\mathcal{C}_x$ to be the open connected component containing $x$.

\begin{definition}
  \label{e:connects}
  Given sets $A, A' \subseteq V$, we say that a path $\sigma = (x_0, \dots, x_l)$ connects $A$ and $A'$ if $x_0 \in B(A,1)$ and $x_l \in B(A', 1)$.
  Note that the point $x_0$ need not be in the set $A$ itself, as it could be solely a neighbor of $A$ (analogously, $x_l$ need not be in $A'$).
\end{definition}

The next lemma shows that under the condition $\mathcal{L}(d_i, c_i)$, any two sets can be joined by a reasonable number of disjoint paths.
Its proof will be a direct consequence of the Max-flow Min-cut Theorem.

\begin{lemma}
  \label{l:disjoint_paths}
  Suppose that $G$ satisfies $\mathcal{L}(d_i, c_i)$ and take disjoint sets $A, A' \subseteq B(x, r)$, where $x \in V$ and $r \geq 1$.
  Then there exist at least $\big\lceil c_i (|A| \wedge |A'|)^{\frac{d_i-1}{d_i}} \big\rceil$ disjoint paths contained in $B(x, r)$, connecting $A$ to $A'$.
\end{lemma}

Consider the graph $G_B = (V_B, \mathcal{E}_B)$ induced by the ball $B = B(x,r)$ in $G$.
We say that $\mathcal{C} \subseteq \mathcal{E}_B$ is a cut-set between $A$ and $A'$ (subsets of $B$) if there is no path in $G_B$ connecting $A$ to $A'$ and avoiding all the edges in $\mathcal{C}$.
We say that such $\mathcal{C}$ is minimal if it has minimal cardinality among all the possible cut-sets between $A$ and $A'$.

\begin{proof}
  Throughout this proof we are going to restrict ourselves to the sub-graph induced by $B(x, r)$ in $G$.
  Let $\mathcal{C} \subseteq \mathcal{E}_B$ be a minimal cut-set between $A$ and $A'$, in this induced sub-graph.
  Define $D$ to be the set of points $y \in B(x, r)$ such that $A$ can be joined to $y$ by a path in $B(x, r)$, without using any edge in $\mathcal{C}$.
  Analogously we define $D'$ replacing the role of $A$ by $A'$.

  Clearly $D \cap D' = \varnothing$ since otherwise one would be able to connect $A$ to $A'$ in $B(x, r)$ without using any edge in $\mathcal{C}$.
  Thus, either $D$ or $D'$ has volume smaller or equal to $|B(x, r)|/2$ and without loss of generality we assume it to be $D$.
  Applying the property $\mathcal{L}(d_i, c_i)$ for the set $D$, we obtain that
  \begin{equation}
    |\mathcal{C}| \geq |\partial_B D| \geq c_i |D|^{\frac{d_i-1}{d_i}} \geq c_i (|A| \wedge |A'|)^{\frac{d_i-1}{d_i}}.
  \end{equation}
  Applying the Max-flow Min-cut Theorem, we conclude the proof of the lemma.
\end{proof}

Let us now comment on our specific choice of isoperimetric inequalities in Definition~\ref{d:isoperimetric}.
In particular, let us contrast it with the more traditional definition \eqref{e:standard_iso} which is clearly weaker than Definition~\ref{d:isoperimetric}.

\begin{remark}
  \label{r:isoperimetric}
  $(a)$ It is important to mention that in the proof of Theorems~\ref{t:bernoulli} and \ref{t:main} we don't use the condition $\mathcal{L}(d_i, c_i)$ directly.
  In fact we only use the existence of several disjoint paths between $A$ and $A'$ as stated in Lemma~\ref{l:disjoint_paths}.

  \vspace{4mm}
  $(b)$ We have already noticed that $\mathbb{Z}^d$ satisfies \eqref{e:standard_iso}, as it follows for instance from Theorem~6.31 of \cite{LP11}, p. 210.
  We now show that $\mathbb{Z}^d$ also satisfies the local isoperimetric inequality $\mathcal{L}(d, c)$.
  First we apply Theorem~5.2 of \cite{MR1318794} to conclude that if $K \subseteq \mathbb{R}^d$ is convex and bounded, then
  \begin{equation}
    \Vol_{d-1} (\partial S \cap K) \geq c \frac{\Vol_d(S)}{\diam(K)},
  \end{equation}
  for every open set $S \subseteq K$ with smooth boundary and $\Vol(S) \leq (3/4)\Vol(K)$.
  If we choose $S$ to be the union of cubes of side length $1$, centered in points of $A \subseteq K \cap \mathbb{Z}^d$, we obtain that $|\partial_K A| \geq c |A|/\diam(K)$.
  Here we should take care to guarantee that all the cubes composing $S$ have a positive proportion of their volume inside $K$.
  As well as the faces of cubes in $S$ corresponding to edges in $\partial_K A$.

  To finish, we partition the ball $\bar K = \{(x_1, \dots, x_d) \in \mathbb{R}^d; \sum_i |x_i| \leq n\}$ into convex sets $K_i$ such that $\diam(K_i) \leq (100 \Vol_{d}(S))^{1/d}$ and $\Vol_d(K_i) \geq 10 \Vol_d(S)$, obtaining $|\partial_{\bar K} A| \geq c \Vol_{d-1} (\partial S \cap \bar K) \geq$ $\sum_i c \Vol_{d-1}(\partial S \cap K_i) \geq$ $c \sum_i \Vol_d(S \cap K_i) \Vol_d(S)^{-1/d} \geq$ $c \Vol_d(S \cap K)^{(d-1)/d} \geq$ $c |A|^{(d-1)/d}$.
  This finishes the proof that $\mathbb{Z}^d$ satisfies $\mathcal{L}(d, c)$.

  We also point out that any graph that is quasi-isometric to $\mathbb{Z}^d$ satisfies $\mathcal{L}(d,c)$.
  This is a consequence of Lemma~4.5 and the Appendix of \cite{kanai1985}.

  \vspace{4mm}
  $(c)$ We now give two examples that satisfy the standard but not the local isoperimetric inequality.
  The first example consists in the infinite regular tree, which clearly satisfies \eqref{e:standard_iso}, for any dimension.
  On the other hand, given any two connected subsets $A$ and $A'$ of the infinite regular tree, there cannot exist two or more disjoint paths connecting $A$ and $A'$.
  Therefore the conclusion of Lemma~\ref{l:disjoint_paths} does not hold, so that infinite regular trees do not satisfy the local isoperimetric inequality of Definition~\ref{d:isoperimetric}.

  On the other hand we know that there exists a phase transition for percolation on regular trees, using for instance a Galton-Watson type argument.
  An indirect way to see that regular trees do not satisfy the local isoperimetric inequality is to observe that the uniqueness of the infinite cluster derived in Theorem~\ref{t:bernoulli} is not satisfied for regular trees, no matter the parameter $p \in (p_c, 1)$.

  Our second example has polynomial growth and is given by two copies of $\mathbb{Z}^d$ connected by a single edge.
  More precisely let $G = (V, \mathcal{E})$, with $V = V_1 \cup V_2$, where $V_1, V_2$ are two disjoint copies of $\mathbb{Z}^d$.
  The edges $\mathcal{E}$ are given by $\mathcal{E}_1 \cup \mathcal{E}_2 \cup \{e\}$, where $\mathcal{E}_i$ connects nearest neighbors vertices in $V_i$, $i = 1, 2$, and $e$ links the origins of $V_1$ and $V_2$.
  The graph $G$ clearly satisfies \eqref{e:standard_iso} with dimension $d$.
  To see this, note that any set $A \subset V_1 \cup V_2$ has at least half of its edges in either $V_1$ or $V_2$, then observe that $\mathbb{Z}^d$ satisfies \eqref{e:standard_iso} with dimension $d$.
  On the other hand, observe that $G$ does not satisfy \eqref{e:isoperimetric} as one can see by taking $x$ to be the origin in $V_1$ and $A = V_1 \cap B(x, r)$ (which gives $|\partial_B A| = 1$ for every $r$).

  We believe it would be an interesting problem to investigate further the relation between these two isoperimetric inequalities, see also Remark~\ref{r:questions}.
\end{remark}

The next result shows that the lower bound $\mathcal{V}_l(d_l, c_l)$ on the volume of balls in $G$ can be obtained from the local isoperimetric inequality, given that $G$ has bounded degree.

\newconst{c:lower_growth}
\begin{lemma}
  \label{l:volume_lower}
  If $G = (V, \mathcal{E})$ is an infinite graph satisfying $\mathcal{L}(d_i, c_i)$ (with $d_i > 1$) and every vertex in $V$ has degree at most $\Delta$, then $G$ also satisfies $\mathcal{V}_l(d_i, \useconst{c:lower_growth})$ for some constant $\useconst{c:lower_growth} = \useconst{c:lower_growth}(d_i, c_i, \Delta)$.
\end{lemma}

See also Lemma~(4.13) of \cite{W00}, p. 45.

\begin{proof}
  Denoting by $S(x, r)$ the set $\{y \in V; d(y,x) = r\}$, we get
  \begin{equation}
    \label{e:lower_bound_B}
    |B(x,r)| \geq \sum_{j=1}^{\lfloor r \rfloor} |S(x,j)| \geq \sum_{j=0}^{\lfloor r \rfloor - 1} \frac{| \partial B(x,j) |}{\Delta} \geq \sum_{j=0}^{\lfloor r \rfloor - 1} \frac{c_i | B(x,j) |^\frac{d_i - 1}{d_i}}{\Delta}.
  \end{equation}
  We want to show by induction that $|B(x,j)|$ has volume at least of order $j^{d_i}$. Choose
  \begin{equation}
    \useconst{c:lower_growth} = 1 \wedge \frac{c_i}{2^{d_i} d_i \Delta}
  \end{equation}
  and observe that $|B(x,1)| \geq 1 \geq (\useconst{c:lower_growth} \cdot 1)^{d_i}$.
  We now suppose that for some $j' \geq 2$ we have $|B(x,j)| \geq (\useconst{c:lower_growth} j)^{d_i}$, for every $j < j'$ and estimate, using \eqref{e:lower_bound_B},
  \begin{equation}
    \begin{split}
      |B(x,j')| & \geq \sum_{j=0}^{j'-1} \frac{c_i | B(x,j) |^\frac{d_i - 1}{d_i}}{\Delta} \geq c_i \sum_{j=0}^{j'-1} \frac {(\useconst{c:lower_growth} j)^{d_i - 1}}{\Delta} \geq \frac{c_i \useconst{c:lower_growth}^{d_i-1}}{\Delta} \sum_{j=0}^{j'-1} j^{d_i-1}\\
      &  \geq \frac{c_i \useconst{c:lower_growth}^{d_i-1}}{\Delta} \frac{(j' - 1)^{d_i}}{d_i} \geq (2 \useconst{c:lower_growth}(j'-1))^{d_i} \overset{j' \geq 2}\geq (\useconst{c:lower_growth}j')^{d_i}.
    \end{split}
  \end{equation}
  Finishing the proof of the lemma by induction on $j'$.
\end{proof}

The next lemma helps us cover a ball of $G$ with not too many balls of smaller radius.
This is crucial in \eqref{e:inductive_bernoulli} to bound the number of ways in which the separation events can propagate to smaller scales.

\newconst{c:Kset}
\begin{lemma}
  \label{l:Kset}
  Suppose that $G = (V, \mathcal{E})$ satisfies $\mathcal{V}_u(d_u, c_u)$ and $\mathcal{V}_l(d_l, c_l)$.
  Then, for any $d < d_l$, there exists a constant $\useconst{c:Kset} = \useconst{c:Kset}(d_l, c_l, d_u, c_u, d)$ such that
  \begin{display}
    \label{e:Kset}
    for every $x \in V$, $r \geq 1$ and $s \in [(\log r)^{\tfrac{2}{d_l - d}}, r/6]$, there exists $K \subseteq B(x, \tfrac{5r}{6})$\\
    such that $|K| \leq \useconst{c:Kset} \frac{r^{d_u}}{s^d}$ and $B(K, \tfrac{s}{6})$ covers $B(x, \tfrac{4r}{6})$.
  \end{display}
\end{lemma}

The proof will make use of the probabilistic method to show the existence of $K$.

\begin{proof}
  Let us first choose $p = s^{-d}$ and consider $(Z_y)_{\smash{y \in B(x, 5r/6)}}$ to be i.i.d. random variables with Bernoulli distribution of parameter $p$.

  We introduce the set of ones $\mathcal{K} = \{y \in B(x, \tfrac{5r}{6}); Z_y = 1\}$ and observe that, given $z \in B(x, \tfrac{4r}{6})$, calling $c_l' = c_l/6$,
  \begin{equation}
    \label{e:bound_K_cover}
    \begin{split}
      P[z \not \in & B(\mathcal{K}, s/6)] = (1-p)^{|B(z,s/6)|} \leq (1-p)^{c_l' s^{d_l}}\\
      & \leq \exp\{-c_l' s^{d_l} p\} = \exp\{-c_l' s^{d_l - d}\},
    \end{split}
  \end{equation}
  since $s \geq (\log r)^{\tfrac{2}{d_l - d}}$, for $r \geq c(d_l, c_l', d_u, c_u, d)$, one obtains
  \begin{equation}
      P[\exists z \in B(x, 4r/6); z \not \in B(\mathcal{K}, s/6)] \leq c_u r^{d_u} \exp\{-c_l' \log^2 r\} \leq 1/3.
  \end{equation}
  We now turn to the bound on $|\mathcal{K}|$.

  Observe that $E[|\mathcal{K}|] = |B(x,5r/6)|s^{-d} \leq c_u (5r/6)^{d_u} s^{-d}$, so that if $\useconst{c:Kset} > c(d_u, c_u)$,
  \begin{equation}
    \label{e:bound_K_size}
    \begin{split}
      P\big[|\mathcal{K}| & > \useconst{c:Kset} r^{d_u} s^{-d} \big] \leq P\big[|\mathcal{K}| > 2E[|\mathcal{K}|]\big] \leq \frac{\Var(|\mathcal{K}|)}{E[|\mathcal{K}|]^2}\\
      & \leq \frac{s^{-d}(1-s^{-d})|B(x,5r/6)|}{s^{-2d}|B(x,5r/6)|^2} \leq c(c_l) s^d r^{-d_l} \overset{s \leq r}\leq c(c_l) r^{d-d_l},
    \end{split}
  \end{equation}
  which is smaller or equal to $1/3$ for $r > c(d_l, c_l, d)$.

  Joining \eqref{e:bound_K_cover} and \eqref{e:bound_K_size}, we get that for $r \geq c(d_l, c_l, d_u, c_u, d)$, there exists a set $K \in B(x, 5r/6)$ such that the conditions in \eqref{e:Kset} hold.
  By possibly increasing the constant $\useconst{c:Kset}$ we can assure that the statement of the lemma also holds for the finitely many values of $r$ that have not been covered above.
  This finishes the proof of the lemma.
\end{proof}

\section{Separation events and renormalization}
\label{s:separation}

In this section we give the main building blocks of the proof of Theorems~\ref{t:bernoulli} and \ref{t:main}.
We start by introducing a definition that traces back from Definition~3.1 of \cite{T10}.

\begin{definition}
  \label{d:SxL}
  Given a configuration $\omega \in \{0,1\}^V$, we define the separation event
  \begin{equation}
    \label{e:SxL}
    \mathcal{S}(x,L) = \Bigg[
    \begin{array}{c}
      \text{there exist connected sets $A, B \subseteq B(x,3L/6)$}\\
      \text{with $d(A,B) > 1$, having diameters at least $L/100$ and}\\
      \text{such that no open path in $B(x, 6L/6)$ connects $A$ to $B$}
    \end{array}
    \Bigg].
  \end{equation}
  Recall Definition~\ref{e:connects} and see Figure~\ref{f:six_balls} for an illustration of this event.
\end{definition}

We would like to stress that the numbers $100$ and $6$ don't have a very important meaning.
Several other choices would lead to valid proofs as well.
Another important observation is that we write the fractions $1/6, 2/6, \dots, 6/6$ without simplifying the numerators and denominators so that the reader can readily see their order.

We intend to analyze the probability of the separation events $\mathcal{S}(x,L)$ as $L$ grows and we do this through a renormalization argument.
For this, fix $\gamma > 1$ and let us introduce
\begin{equation}
  \label{e:Lk}
  L_0 = 10000 \qquad \text{and} \qquad L_{k+1} = L_k^\gamma, \text{ for $k \geq 0$.}
\end{equation}
Let us observe that the value of $L_0$ is not important in what follows and that the sequence $L_k$ grows much faster than exponential, in fact $L_k = L_0^{\gamma^k}$.

A very important property of the separation events defined above is that they behave well with respect to scale change.
More precisely, in the next lemma we will show that the occurrence of the event $\mathcal{S}(x,L_{k+1})$ implies the occurrence of similar events in the previous scale $L_k$.

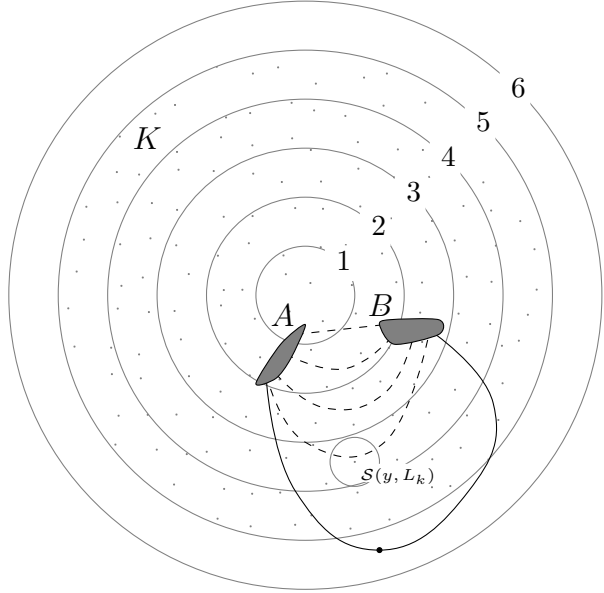
\begin{figure}[h]
  \centering
  \begin{tikzpicture}[scale=0.65]
    \begin{scope}
      \clip (6,6) circle (5);
      \foreach \x in {1,...,16}
      { \foreach \y in {1,...,16}
        \draw[color=gray,fill=gray] (0.7 * \x + 0.2 * rand, 0.7 * \y + 0.2 * rand) circle (0.01);
      }
    \end{scope}
    \draw[fill, color=white] (2.4,9.6) rectangle (3.2,8.8);
    \node at (2.8,9.2) {$K$};
    \draw[fill, color=white] (6.5,3.1) rectangle (7.5,2.1);
    \foreach \x in {1,...,6}
    { \draw[color=gray] (6, 6) circle (\x);
      \draw[fill, color=white](5.7 + 0.7071 * \x, 5.6 + 0.7071 * \x) rectangle (6.3 + 0.7071 * \x, 6.3 + 0.7071 * \x);
      \node[right] at (5.7 + 0.7071 * \x, 6 + 0.7071 * \x) {\small $\x$};
    }
    \draw[color=gray] (7,2.6) circle (.5);
    \draw[fill, color=gray] (7,2.6) circle (.01);
    \draw[fill, color=white] (7.1,2.55) rectangle (8.5,2.2);
    \node[below right] at (6.9,2.7) {\tiny $\mathcal{S}(y,L_k)$};
    \draw[dashed] plot [smooth, tension=2] coordinates {(5.8,5.2) (7.6,5.4)};
    \draw[dashed] plot [smooth, tension=.8] coordinates {(5.6,4.8) (6.8,4.5) (7.8,5.2)};
    \draw[dashed] plot [smooth, tension=.8] coordinates {(5.4,4.4) (6.4,3.7) (7.6,4) (8.2,5.2)};
    \draw[dashed] plot [smooth, tension=.8] coordinates {(5.2,4.4) (6,3) (7.6,3) (8.5,5.2)};
    \draw plot [smooth, tension=0.8] coordinates {(5.2,4.4) (5.8,2) (7.5,0.8) (9.3,2) (9.8,3.8) (8.5,5.3)};
    \draw[fill=gray] plot [smooth cycle] coordinates {(6,5.4) (5.5,5) (5,4.2) (5.4,4.3) (5.7, 4.7)};
    \draw[fill=gray] plot [smooth cycle] coordinates {(7.5,5.4) (7.8,5) (8.7,5.2) (8.7,5.5) (7.7, 5.5)};
    \node[left] at (6,5.6) {$A$};
    \node[left] at (8,5.8) {$B$};
    \draw[fill] (7.5,0.8) circle (0.05);
  \end{tikzpicture}
  \caption{The six balls $B(x,L_{k+1}/6), \dots, B(x,6L_{k+1}/6)$.
    The sets $A$ and $B$ from the definition of $S(x,L_{k+1})$ are pictured, together with a solid path connecting them.
    According to the definition of $S(x,L_{k+1})$, this solid path must pass through a closed vertex.
    The gray dots in the picture represent the set $K$ from Lemma~\ref{l:Kset}, while the dashed paths between $A$ and $B$ illustrate the statement of Lemma~\ref{l:disjoint_paths}.
    We also indicate the occurrence of the event $\mathcal{S}(y,L_k)$ as in Lemma~\ref{l:string}.
  }
  \label{f:six_balls}
\end{figure}

\newconst{c:SxL}
Let us pick $\useconst{c:SxL}$ large enough so that for $k \geq \useconst{c:SxL}$ we have
\begin{equation}
  \label{e:const_SxL}
  L_k < L_{k+1}/2000.
\end{equation}
This constant will be useful in the next lemma.

\begin{lemma}
  \label{l:string}
  Fix $x \in V$ and $k \geq \useconst{c:SxL}$ and assume the occurrence of the separation event $\mathcal{S}(x,L_{k+1})$.
  Consider a set $K \subseteq B(x,5L_{k+1}/6)$ such that $B(K, L_k/6)$ covers $B(x,4L_{k+1}/6)$ and a pair $A, B$ such that
  \begin{enumerate}[\quad \; a)]
  \item $A$ and $B$ are connected and contained in $B(x, 3L_{k+1}/6)$,
  \item their diameters are greater or equal to $L_{k+1}/1000$ and
  \item no open path in $B(x, 6L_k/6)$ connects $A$ and $B$.
  \end{enumerate}
  Note the similarity between these conditions and the ones in Definition~\ref{d:SxL}, see also Remark~\ref{r:1000}.
  Then, for every path $\sigma$ in $B(x, 4L_{k+1}/6)$ connecting $A$ and $B$, there exists $y \in K$ such that
  \begin{enumerate}[\quad i)]
  \item $\sigma$ intersects $B(y, L_k/6)$ and
  \item the event $\mathcal{S}(y,L_k)$ holds.
  \end{enumerate}
\end{lemma}
See Figure~\ref{f:six_balls} for an illustration of the above lemma.

\begin{remark}
  \label{r:1000}
  Note that we ask the diameters of $A$ and $B$ to be at least $L_{k+1}/1000$, which is a weaker requirement than that of the definition of $\mathcal{S}(x, L_{k+1})$.
  The need for this will become clear in \eqref{e:AprimeBprime}, see the proof of Lemma~\ref{l:cascade}.
\end{remark}

\begin{proof}
  We define the following modification of the original configuration $\omega$
  \begin{equation}
    \bar \omega(x) =
    \begin{cases}
      0 \quad & \text{if $x \not \in B(x, 6 L_{k+1} / 6)$},\\
      \omega(x) & \text{otherwise}.
    \end{cases}
  \end{equation}
  We define $\bar{\mathcal{C}}_z$ to be the open component containing $z$ under the configuration $\bar \omega$.
  We also define
  \begin{equation}
    \bar{\mathcal{C}}_A = A \cup \bigcup_{z \in A} \bar{\mathcal{C}}_z.
  \end{equation}
  Note that $\bar{\mathcal{C}}_A \setminus A$ is open.

  For a path $\sigma$ as in the statement of the lemma, we denote its points by $x_0, x_1, \dots, x_l$, where $x_0 \in B(A,1)$ and $x_l \in B(B,1)$.
  Let us introduce
  \begin{equation}
    i_o = \min \{i = 0, \dots, l; B(x_i, L_k/6) \cap \bar{\mathcal{C}}_A = \varnothing\},
  \end{equation}
  where we take $i_o = l$ if the above set is empty.

  Since $x_{i_o} \in B(x, 4L_{k+1}/6) \subseteq B(K,L_k/6)$, there exists some $y \in K$ such that $x_{i_o} \in B(y, L_k/6)$.
  In particular $\sigma$ intersects $B(y,L_k/6)$.

  All we have to show now is that the event $\mathcal{S}(y, L_k)$ holds and we will do this splitting the proof in two distinct cases.

  {\bf Case 1:} $B(x_i,L_k/6) \cap \bar{\mathcal{C}}_A \neq \varnothing$, for every $i = 1, \dots, l$.

  \noindent In this case by our definition, $i_o = l$ so that $x_{i_o} \in B(B,1)$.
  This implies that
  \begin{enumerate}
  \item $B(y,2 L_k/6)$ intersects $\bar{\mathcal{C}}_A$ (since this ball contains $B(x_l, L_k/6)$) and
  \item $B(y,2 L_k/6)$ intersects $B$ (via $x_l$).
  \end{enumerate}
  Denote by $x_a$ (respectively $x_b$) an arbitrary point in the intersection of $B(y,2L_k/6)$ and $\bar{\mathcal{C}}_A$ (respectively $B(y, 2 L_k / 6) \cap B$).

  We will now define the sets $A'$ and $B'$ that confirm the occurrence of the event $\mathcal{S}(y, L_k)$.
  For this, consider the modified percolation configuration restricted to $B(y, 3L_k/6)$ and declared open in $A \cup B$, that is
  \begin{equation}
    \omega'(z) =
    \begin{cases}
      0, \quad & \text{if $z \not \in B(y, 3L_k/6)$,}\\
      1, & \text{if $z \in B(y, 3L_k/6) \cap A \cup B$ and}\\
      \omega(z), & \text{if $z \in B(y, 3L_k/6) \setminus A \cup B$.}
    \end{cases}
  \end{equation}
  We then let $A'$ and $B'$ be the open connected components under $\omega'$ containing $x_a$ and $x_b$ respectively.

  To finish this case, all we need to show is that
  \begin{gather}
    \label{e:ABprime_diam}
    \text{$A'$ and $B'$ are connected, have diameter at least $L_k/100$ and}\\
    \label{e:ABprime_separation}
    \text{no path in $B(x, 6L_k/6)$, connecting $A'$ to $B'$ is open.}
  \end{gather}

  To verify \eqref{e:ABprime_diam}, observe first that $A'$ and $B'$ are clearly connected and contained in $B(y, 3L_k/6)$, according to the definition of $\omega'$.
  Note as well that both $A'$ and $B'$ contain a point in $B(y, 2L_k/6)$, namely $x_a$ and $x_b$ respectively.
  Now, since the diameter of both $B$ and $\bar{\mathcal{C}}_A$ are larger or equal to $L_{k+1}/1000$ by hypothesis, they must not be contained in $B(y, 3L_k/6)$ by \eqref{e:const_SxL}.
  That means that both $A'$ and $B'$ must touch the internal boundary of $B(y, 3L_k/6)$, so that their diameters are at least $L_k/6$, proving \eqref{e:ABprime_diam}.

  We now turn to the proof of \eqref{e:ABprime_separation}.
  For this, let $\sigma'$ denote a path in $B(y, 6L_k/6)$ connecting $A'$ and $B'$.
  By the definition of $A'$ and $B'$, we can extend $\sigma'$ to a path $\sigma''$ connecting $B$ and $\bar{\mathcal{C}}_A$ without leaving $B(y, 6 L_k / 6)$ and only adding sites which are open in the $\omega$ configuration.
  Since $\sigma''$ is contained in $B(y,L_k) \subseteq B(x, L_{k+1})$, we know by $\mathcal{S}(x,L_{k+1})$ that $\sigma''$ cannot be open.
  This means that $\sigma'$ was not open to start with.
  This proves \eqref{e:ABprime_separation}, finishing the proof that $\mathcal{S}(y, L_k)$ holds in this Case~1.

  {\bf Case 2:} $B(x_{i_o},L_k/6) \cap \bar{\mathcal{C}}_A = \varnothing$.

  \noindent In this case, we know that $i_0 \geq 1$ (since $x_0 \in B(A,1)$ and $A \subseteq \bar{\mathcal{C}}_A$).
  Moreover, by the minimality of $i_o$, we have $B(x_{i_o - 1}, L_k/6) \cap \bar{\mathcal{C}}_A \neq \varnothing$, implying that $B(y, 2L_k/6 + 1) \cap \bar{\mathcal{C}}_A \neq \varnothing$.
  We pick $x_a$ to be an arbitrary point in this intersection.

  As in the previous case, we need to define the sets $A'$ and $B'$ that guarantee the occurrence of $\mathcal{S}(y,L_k)$.
  For this we define the modified percolation configuration restricted to $B(y, 3L_k/6)$ and open in $A$, that is
  \begin{equation}
    \omega'(z) =
    \begin{cases}
      0, \quad & \text{if $z \not \in B(y, 3L_k/6)$,}\\
      1, & \text{if $z \in B(y, 3L_k/6) \cap A$ and}\\
      \omega(z), & \text{if $z \in B(y, 3L_k/6) \setminus A$.}
    \end{cases}
  \end{equation}
  We then define $A'$ to be the connected component under $\omega'$ containing $x_a$ and $B' = B(x_{i_o}, L_k/6 - 1)$.
  It is clear that the diameters of both $A'$ and $B'$ are at least $L_k/100$ (here for $A'$ we use the same argument as in \eqref{e:ABprime_diam}).
  Therefore, all we need to show is \eqref{e:ABprime_separation} and for this we consider any path $\sigma'$ in $B(y,L_k)$ connecting $A'$ and $B'$.
  Should $\sigma'$ be open, then we would have $\bar{\mathcal{C}}_A$ neighboring $B'$, which contradicts the assumption defining Case~2.
  This shows that $\mathcal{S}(y,L_k)$ holds.

  Joining the two cases, we have shown that $\mathcal{S}(y, L_k)$ holds, finishing the proof of the lemma.
\end{proof}

We now use the isoperimetry of the graph $G$, together with the above lemma to show that the separation event $\mathcal{S}(x,L_{k+1})$ induces the occurrence of several separation events at the smaller scale $k$.
This cascading property is the main ingredient in the recursion inequalities leading to our main results.

\newconst{c:cascade}
\begin{lemma}
  \label{l:cascade}
  For a graph $G$ satisfying $\mathcal{L}(d_i, c_i)$ and $\mathcal{V}_u(d_u, c_u)$, fix $x \in V$, $k > \useconst{c:SxL}$ and assume that $S(x, L_{k+1})$ holds.
  Let us also fix a set $K \subseteq B(x, 5L_{k+1}/6)$ such that $B(K, L_k/6)$ covers $B(x, 4L_{k+1}/6)$.
  Then, there exist at least
  \begin{equation}
    \label{e:N}
    N := \Big\lfloor \useconst{c:cascade} L_k^{\gamma \big(\tfrac{d_i - 1}{d_i}\big) - (d_u - 1)} \Big\rfloor \text{ many points $y_j \in K$ such that $\mathcal{S}(y_j, L_k)$ holds.}
  \end{equation}
  Where $\useconst{c:cascade} = \useconst{c:cascade}(d_i, c_i, d_u, c_u)$.
  Moreover we can assume that $d(y_j, y_{j'}) \geq 3 L_k$ for every $1 \leq j < j' \leq N$.
\end{lemma}

It is important to observe that the above lemma is useless unless $d_i > 1$, which encompasses the intuition that the dimension of $G$ should be larger than one for percolation to take place.

\begin{proof}
  Since $\mathcal{S}(x, L_{k+1})$ holds, there exist $A, B \subseteq B(x, 3L_{k+1}/6)$ which are connected, have diameter at least $L_{k+1}/100$ and are separated in $B(x, 6L_{k+1}/6)$.
  Note that the distance between $A$ and $B$ could be as small as $2$, which would in fact weaken our arguments.
  As a first step in the proof we show that
  \begin{display}
    \label{e:AprimeBprime}
    there exist $A', B' \subseteq B(x, 3L_{k+1}/6)$ which are connected, have diameters at least $L_{k+1}/1000$, are separated in $B(x, 6L_{k+1}/6)$ and $d(A', B') \geq L_{k+1}/750$.
  \end{display}

  To see why this is the case, take any $a \in A$ and observe that, since the diameter of $B$ is at least $L_{k+1}/100$, it cannot be contained in $B(a, L_{k+1}/300)$.
  This way, let us pick an arbitrary $b \in B \setminus B(a, L_{k+1}/300)$.
  Since $A$ and $B$ are connected and have diameter at least $L_{k+1}/100$, we can find connected sets $A' \subseteq A$ and $B' \subseteq B$, contained respectively in $B(a, L_{k+1}/300)$ and $B(b, L_{k+1}/300)$ satisfying \eqref{e:AprimeBprime}.

  Given $A'$ and $B'$ as in \eqref{e:AprimeBprime}, we note that their volumes are bounded from below by their diameters, so that we can use Lemma~\ref{l:disjoint_paths} to obtain that
  \begin{display}
    there exist $N' = \Big\lceil c_i (L_{k+1}/1000)^{\frac{d_i - 1}{d_i}} \Big\rceil$ many disjoint paths $\sigma_1, \dots, \sigma_{N'}$, contained in $B(x, 4L_{k+1}/6)$ and connecting $A'$ to $B'$.
  \end{display}

  We now use Lemma~\ref{l:string} to conclude that there exist points $y'_1, \dots, y'_{N'} \in K$, such that
  \begin{enumerate} [\quad a)]
  \item $\sigma_i$ intersects $B(y'_i, L_k/6)$ and
  \item the event $\mathcal{S}(y'_i, L_{k})$ holds,
  \end{enumerate}
  for every $i \leq N'$.
  Note that the points $y'_1, \dots, y'_{N'}$ need not be distinct, only the paths $\sigma'_i$'s need be disjoint.

  We now have to verify that we can extract from $\{y'_i\}_{i=1}^{N'}$ a subset of $N$ points which are mutually far apart.
  For this, recall that the distance from $A'$ to $B'$ is at least $L_{k+1}/750$ (see \eqref{e:AprimeBprime}).
  This means that the diameter of each $\sigma_i$ must be at least $L_{k+1}/750$, which is larger than $L_k$ by \eqref{e:const_SxL}.
  This way we conclude that whenever a path $\sigma_i$ intersects $B(y'_i, L_k/6)$ it must intersect at least $L_k/6$ points in $B(y_i, 4L_k)$.
  Since all the paths $\sigma_i$ are disjoint,
  \begin{display}
    the number of paths $(\sigma_j)_{j \leq N'}$ that can intersect a given $B(y_i, 4L_k)$ is at most\\
    $\displaystyle \frac{6c_u (4 L_k)^{d_u}}{L_k} \leq c(d_u, c_u)L^{d_u-1}_k$.
  \end{display}
  This way, a simple counting argument shows that we can choose $\useconst{c:cascade}(d_i, c_i, d_u, c_u)$, such that there must be at least $N$ points $(y_i)_{i \leq N}$ chosen within the $y'_i$'s such that $d(y_i,y_{i'}) \geq 3L_k$, as required in the statement.
\end{proof}

In the next section we prove the main results of this article.

\section{Proofs of main results}
\label{s:proofs}

We start by proving Theorem~\ref{t:bernoulli}.
Although this result can be derived directly from Theorem~\ref{t:main}, we understood that giving its proof separately is a good warm up for the dependent case.

\begin{proof}[Proof of Theorem~\ref{t:bernoulli}]
  For the proof, we fix an arbitrary $d \in (0, d_l)$.
  Since $d_i > 1$, we can pick $\gamma > 1$ such that
  \begin{equation}
    \gamma \big( \frac{d_i - 1}{d_i} \big) > d_u - 1,
  \end{equation}
  which is then used in the definition of $L_k$ in \eqref{e:Lk}.

  The main step of the proof is to prove a fast decay for
  \begin{equation}
    \label{e:p_k}
    p_k = \sup_{x \in V} \mathbb{P}[\mathcal{S}(x, L_k)],
  \end{equation}
  which is done by induction.

  \newconst{c:s_logr}
  Given $k \geq 1$ and $x \in V$ we are going to explore the consequences of $\mathcal{S}(x, L_{k+1})$.
  We first apply Lemma~\ref{l:Kset} with $r = L_{k+1}$ and $s = L_k$.
  For $k \geq \useconst{c:s_logr} = \useconst{c:s_logr}(\gamma, d_l, d)$, we have $s \in [(\log r)^{2/(d_l - d)}, r/6]$.
  Lemma~\ref{l:Kset} gives us the existence of a set $K \subseteq B(x, 5L_{k+1}/6)$ such that
  \begin{enumerate}[\quad a)]
  \item $|K| \leq \useconst{c:Kset} \frac{L_{k+1}^{d_u}}{L_k^d} = \useconst{c:Kset} L_k^{\gamma d_u - d}$ and
  \item $B(x, 4L_{k+1}/6) \subseteq B(K, L_k/6)$.
  \end{enumerate}

  Our induction will rely on the cascading property of the separation events $\mathcal{S}(x, L_{k+1})$, as described in Lemma~\ref{l:cascade}.
  More precisely, we will use that $\mathcal{S}(x,L_{k+1})$ implies the occurrence of $J$ separation events at the smaller scale $k$.
  To choose $J$, pick some $\beta$ such that
  \begin{equation}
    \label{e:beta_bernoulli}
    \beta > (\gamma d_u - d) \vee \gamma(1 + \chi)
  \end{equation}
  then let $J \geq 2$ be an integer such that
  \begin{equation}
    \label{e:choose_J}
    J > \frac{\gamma \beta}{\beta - (\gamma d_u - d)}.
  \end{equation}
  \newconst{c:2points}
  We can now choose $\useconst{c:2points} = \useconst{c:2points}(d_i, c_i, d_u, c_u, \gamma, \beta, J) > \useconst{c:s_logr}$ such that for $k \geq \useconst{c:2points}$ we have
  \begin{equation}
    \Big\lfloor \useconst{c:cascade} L_k^{\gamma \big( \tfrac{d_i - 1}{d_i}\big) - (d_u - 1)} \Big\rfloor > J.
  \end{equation}
  We know by Lemma~\ref{l:cascade} that
  \begin{equation}
    \mathbb{P}[\mathcal{S}(x, L_{k+1})] \leq \mathbb{P}\Big[
    \begin{array}{c}
      \text{there exist $y_1, \dots, y_{J} \in K$, within distance $3L_k$}\\
      \text{and such that $\mathcal{S}(y_i, L_k)$ holds for all $i = 1, \dots, J$}
    \end{array}
    \Big],
  \end{equation}
  so that
  \begin{equation}
    \label{e:inductive_bernoulli}
    p_{k+1} \leq |K|^J p_k^J \leq \big(\useconst{c:Kset} L_k^{\gamma d_u - d}\big)^J p_k^J.
  \end{equation}
  We are going to show that
  \begin{equation}
    \label{e:pk_decay}
    p_k \leq L_k^{-\beta} \text{ for $k$ large enough}.
  \end{equation}

  \newconst{c:ratio_one}
  Suppose first that $p_k \leq L_k^{-\beta}$ and use \eqref{e:inductive_bernoulli} to estimate
  \begin{equation}
    \frac{p_{k+1}}{L_{k+1}^{-\beta}} \leq \useconst{c:Kset}^J L_k^{J(\gamma d_u - d) - J\beta + \gamma \beta} \leq \useconst{c:Kset}^J L_k^{-{\displaystyle (}J(\beta - (\gamma d_u - d)) - \gamma \beta{\displaystyle )}}.
  \end{equation}
  By the choice of $\beta$ in \eqref{e:beta_bernoulli} and $J$ in \eqref{e:choose_J}, we see that the above is smaller or equal to one for $k \geq \useconst{c:ratio_one} = \useconst{c:ratio_one}(J, \gamma, \beta, d_i, c_i, d_u, c_u, d_l, c_l)$.

  This means that if \eqref{e:pk_decay} holds for a given $k' \geq \useconst{c:ratio_one}$, then it must also hold for all $k > k'$.
  It is clear that as the percolation parameter $p$ converges to one, the probability of $\mathcal{S}(x, L_{k'})$ converges to zero uniformly over $x$ (since $B(x, L_{k'})$ will likely be completely open).
  Therefore, we know that for some $p$ close enough to one \eqref{e:pk_decay} holds .

  To finish the proof of the theorem, one should simply observe that $\beta > \gamma(1 + \chi)$ and employ the Lemma~\ref{l:lego} below.
\end{proof}

The renormalization scheme that we have employed above gives us the decay of the probabilities $p_k$ of the separation events.
The next lemma shows that this is enough to show the existence of an infinite connected component.

\begin{lemma}
  \label{l:lego}
  If for some choice of $\gamma > 1$ and $\beta > \gamma (1 + \chi)$, we have $p_k \leq L_k^{-\beta}$ for all $k \geq \bar k$, then
  \begin{equation}
    \label{e:unique_C_infty}
    \mathbb{P}\big[ \text{there exists a unique infinite open cluster $\mathcal{C}_\infty$} \big]  = 1.
  \end{equation}
  Moreover, for every $x \in V$,
  \begin{equation}
    \label{e:decay_C_x}
    \mathbb{P}\big[ L \leq \diam(\mathcal{C}_x) < \infty \big] \leq c(\gamma, \chi, \bar k) L^{-\chi}.
  \end{equation}
\end{lemma}

\begin{proof}
  We start by fixing a path $\sigma: \mathbb{N} \to V$ satisfying the so-called \emph{half-axis property}, that is $d(\sigma(i), \sigma(j)) = |i - j|$ for all $i, j \in \mathbb{N}$.
  To see why such a path exists, fix a point $x \in V$ and consider geodesic paths $\sigma_{x, y}$ from $x$ to all vertices $y \in V$.
  Then construct $\sigma$ step by step, starting from $x$ and only following edges that have been used by infinitely many geodesics $\sigma_{x,y}$.
  A more detailed proof of this and stronger statements is provided for instance by Theorem~3.1 of \cite{Watkins1986341}.

  Given $\sigma$, we now fix the following collection of points
  \begin{equation}
    x_{k, i} = \sigma(i L_k / 6), \text{ for $k \geq 1$ and $i = 0, \dots, L_{k+1}/L_k$.}
  \end{equation}
  Now we claim that
  \begin{display}
    \label{e:C_infty}
    on the event $\mathcal{G}_{k_o} = \mcap_{k \geq k_o} \mcap_{i=0}^{L_{k+1}/L_k} \mathcal{S}(x_{k,i}, L_k)^c$, there exists a unique, infinite connected component $\mathcal{C}_\infty$. Moreover, either $x \in \mathcal{C}_\infty$ or $\diam(\mathcal{C}_x) \leq L_{k_o}$.
  \end{display}
  Before proving the above statement, let us see why this is enough to establish Lemma~\ref{l:lego}.

  Let us first estimate the probability of $\mathcal{G}_{k_o}^c$, for $k_o \geq \bar k$, by
  \begin{equation}
    \begin{split}
      \mathbb{P}[\mathcal{G}_{k_o}^c] & = \mathbb{P}\Big[ \mcup_{k \geq k_o} \mcup_{i=1}^{L_{k+1}/L_k} \mathcal{S}(x_{k,i}, L_k)
      \Big] \leq \sum\limits_{k \geq k_o} \sum\limits_{i=0}^{L_{k+1}/L_k} p_k\\
      & \leq \sum_{k \geq k_o} L_k^{\gamma - 1 - \beta} \overset{\beta > \gamma(1 + \chi) - 1}\leq \sum_{k \geq k_o} L_k^{-\gamma \chi} = L_{k_o}^{- \gamma \chi} \sum_{k \geq 0} (L_0^{-\gamma \chi})^{\gamma^{k_o}(\gamma^k - 1)}\\
      & \leq L_{k_o}^{- \gamma \chi} \sum_{k \geq 0} (L_0^{-\gamma \chi})^{\gamma^k - 1} \leq c(\gamma, \chi) L_{k_o}^{-\gamma \chi}
    \end{split}
  \end{equation}
  With the above bound and assuming \eqref{e:C_infty}, one gets \eqref{e:unique_C_infty} directly. Moreover, given $L \geq c(\bar k)$, we fix $k_o \geq \bar k$ such that $L_{k_o} \leq L < L_{k_o + 1}$ to estimate
  \begin{equation}
    \mathbb{P}\big[ L < \diam(\mathcal{C}_x) < \infty \big] \leq \mathbb{P}\big[ L_k < \diam(\mathcal{C}_x) < \infty \big] \overset{\eqref{e:C_infty}}\leq c(\gamma, \chi) L_{k}^{-\gamma \chi} \leq c(\gamma, \chi) L^\chi,
  \end{equation}
  as claimed in the statement of the lemma.
  To cover the finite values of $L$ not covered above, we can increase the constant $c(\gamma, \chi)$ to some $c(\gamma, \chi, \bar k)$.

  All we need to prove now is \eqref{e:C_infty}. We claim that for this it is enough to show that
  \begin{display}
    \label{e:one_infinite}
    on $\mathcal{G}_{k_o}$, there exists an infinite connected component\\that touches $B(x, L_{k_o}/30)$.
  \end{display}
  To see why this is enough, we start by observing that the uniqueness of the infinite cluster follows directly from $\mathcal{G}_{k_o}$, since the existence of two or more infinite connected components would trigger the occurrence of $\mathcal{S}(x_{k,0}, L_k) = \mathcal{S}(x, L_k)$ for all but a finite number of $k$'s.
  Moreover it is also a trivial consequence of $\mathcal{G}_{k_o}$ that either $x \in \mathcal{C}_\infty$ or $\diam(\mathcal{C}_x) \leq L_{k_o}$.
  Otherwise, we would have two separated components touching $B(x, L_{k_o}/30)$, with diameters at least $L_{k_o}$, contradicting the fact that $\mathcal{S}(x,L_{k_o})$ did not occur.

  Let us now turn to the proof of \eqref{e:one_infinite}, which will be done by constructing several small paths and joining them using the absence of separation events.

  For now, fix $k \geq k_o$.
  Given any $i = 0, \dots, (L_{k+1}/L_k) - 1$, we can use the fact that we are on $S(x_{k,i}, L_k)^c$ to obtain an open path $\sigma_{k,i} \subseteq B(x_{k,i}, 6L_{k}/6)$ connecting $B(x_{k,i}, L_k/30)$ to $B(x_{k,i+1}, L_k/30)$, see Figure~\ref{f:joining_paths}.

  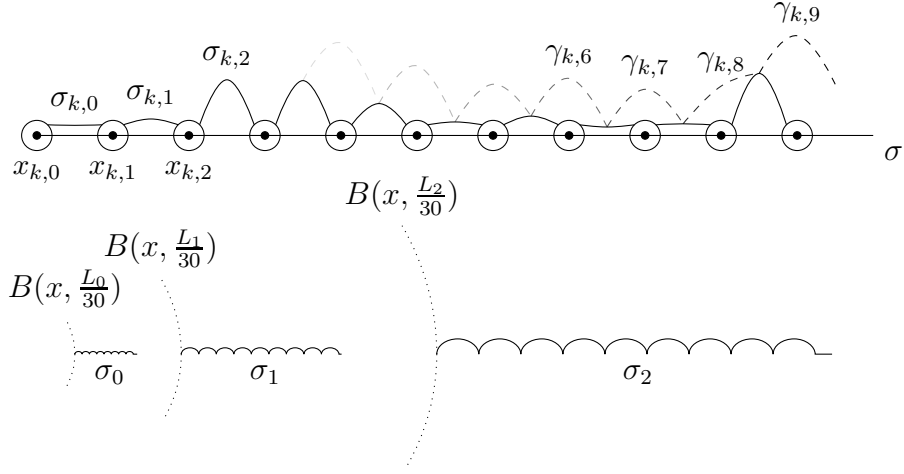
\begin{figure}[htbp]
    \centering
    \begin{tikzpicture}[scale=1]
      \draw (0,0) -- (11,0); \node[below right] at (11,0) {$\sigma$};
      \foreach \x in {0,...,10} {
        \draw[fill] (\x, 0) circle (.05);
        \draw (\x, 0) circle (0.2);
      }
      \foreach \x in {0,...,2} {
        \coordinate (a) at (\x + .5, .4 + .4 * rand);
        \draw plot [smooth, tension=0.8] coordinates {(\x + 0.14142, 0.14142) (a) (\x + 1 - 0.14142, 0.14142)};
        \node[above] at (a) {$\sigma_{k,\x}$};
        \node[below] at (\x,-.2) {\small $x_{k,\x}$};
      }
      \pgfmathparse{.4 + .4 * rand}
      \xdef\z{\pgfmathresult}
      \foreach \x in {3,...,9} {
        \pgfmathparse{.5 + .4 * rand};
        \xdef\y{\pgfmathresult};
        \draw plot [smooth, tension=0.8] coordinates {(\x + 0.14142, 0.14142) (\x + .5, \z) (\x + 1 - 0.14142, 0.14142)};
        \coordinate (a-\x) at (\x + 1, \z + .5);
        \draw[dashed, opacity=(\x-2)/7] plot [smooth, tension=0.8] coordinates {(\x + .5, \z) (a-\x) (\x + 1.5, \y)};
        \xdef\z{\y};
      }
      \foreach \x in {6,...,9} {
        \node[above] at (a-\x) {$\gamma_{k,\x}$};
      }
      \begin{scope}[xshift=-5mm]
        \foreach \k in {0,1,2} {
          \pgfmathparse{2.4^\k};
          \xdef\x{\pgfmathresult};
          \pgfmathparse{1.92*\x};
          \xdef\y{\pgfmathresult};
          \draw[decorate,decoration={bumps,amplitude=\x,segment length=\y mm, pre length=0mm}] plot coordinates {(\x, -2.9) (\y -.1, -2.9)};
          \pgfmathparse{(\x + \y)/2};
          \xdef\z{\pgfmathresult};
          \node[below] at (\z,-2.9) {$\sigma_\k$};
          \draw[dotted] (\x,-2.9) arc (0:30:{1.5* sqrt(\x-.6)}) node[above] {${\tiny B(x,\tfrac{L_\k}{30})}$}; \draw[dotted] (\x,-2.9) arc (0:-30:{1.5*sqrt{\x - .6}});
        }
      \end{scope}
    \end{tikzpicture}
    \caption{Above we see the paths $\sigma_{k,i}$ (for a fixed $k$) and the paths $\gamma_{k,i}$ that join them. Below we can see the paths $\sigma_k$ that were constructed in \eqref{e:sigma_prime_k}.}
    \label{f:joining_paths}
  \end{figure}

  We would like to joint the paths $\sigma_{k,i}$ into a single connected component (recall that the balls $B(x_{k,i}, L_k/30)$ that link them are not necessarily open).
  For this we will use the absence of separation again, but first we need to estimate their diameters
  \begin{equation}
    \diam(\sigma_{k,i}) \geq d(B(x_{k,i}, L_k/30), B(x_{k,i+1}, L_k/30)) \geq d(x_{k,i}, x_{k,i+1}) - \tfrac{L_k}{15} \geq \tfrac{L_k}{20}.
  \end{equation}
  Therefore, we are sure that before $\sigma_{k,i}$ has a chance to exit $B(x_{k,i}, 3L_k/6)$, it already has diameter at least $L_k/30$.
  This way, we can now obtain (again using that we are in $\mathcal{S}(x_{k,i},L_k)^c$) open paths $\gamma_{k,i}$ joining $\sigma_{k,i}$ to $\sigma_{k,i+1}$, for $i = 0, \dots, (L_{k+1}/L_k) - 2$.

  Consider now an open path $\sigma_k$ that visits the ranges of all $\sigma_{k,i}$ and $\gamma_{k,i}$, for $i \leq L_{k+1}/L_k - 1$.
  It is clear that
  \begin{display}
    \label{e:sigma_prime_k}
    $\sigma_k$ has diameter at least $d(x_{k,0}, x_{k,L_{k+1}/L_k}) - L_k/30 \geq L_{k+1}/30$, moreover we can assume that $\sigma_k$ starts from the ball $B(x,L_{k}/30)$.
  \end{display}
  The last step of the proof is to joint all the paths $(\sigma_k)_{k \geq k_o}$ to build an infinite connected component.
  An important remark here is that we could not have attempted to joint all the paths $(\sigma_{k,0})_{k \geq k_o}$ before building the $\sigma_k$'s.
  The reason is that the diameter of the $\sigma_{k,0}$ is comparable to $L_k$ and not $L_{k+1}$ as the $\sigma_k$.

  To finish the proof of \eqref{e:one_infinite}, note that on $\mathcal{G}_{k_o}$, for any $k \geq k_o$, the ranges of $\sigma_k$ and $\sigma_{k+1}$ must be on the same connected component.
  Indeed, since we are outside the event $\mathcal{S}(x, L_{k+1})$, before these paths have a chance to exit $B(x,3L_{k+1}/6)$ they must have already covered a diameter larger or equal to $L_{k+1}/100$.
  This finishes the proof of \eqref{e:one_infinite} and consequently of the lemma.
\end{proof}

We now turn to the proof of Theorem~\ref{t:main}, which deals with dependent percolation models.
But before, we will need a very basic consequence of Definition~\ref{d:decouple}.

\begin{lemma}
  \label{l:several_boxes}
  Suppose that $\mathbb{P}$ on $\{0,1\}^V$ satisfies the decoupling inequality $\mathcal{D}(\alpha, c_\alpha)$.
  Fix any choice of $r \geq 1$, an integer $J \geq 1$ and points $x_1, \dots, x_J \in V$ such that
  \begin{equation}
    \min_{1 \leq i < i' \leq J} d(x_i, x_{i'}) \geq 3r.
  \end{equation}
  Then
  \begin{equation}
    P(\mathcal{G}_1 \cap \dots \cap \mathcal{G}_J) \leq \big(P(\mathcal{G}_1) + c_\alpha r^{-\alpha} \big) \dotsm \big(P(\mathcal{G}_J) + c_\alpha r^{-\alpha} \big),
  \end{equation}
  for any decreasing events $\mathcal{G}_1, \dots, \mathcal{G}_J$ such that $\mathcal{G}_i \in \sigma(Z_y; y \in B(x_i,r))$.
\end{lemma}

\begin{proof}
  The result follows from the simple estimate
  \begin{equation}
    \begin{array}{e}
      P(\mathcal{G}_1 \cap \dots \cap \mathcal{G}_J) & = & P(\mathcal{G}_1 \cap \dots \cap \mathcal{G}_{J-1}) P(\mathcal{G}_J | \mathcal{G}_1 \cap \dots \cap \mathcal{G}_{J-1})\\
      & \overset{\mathcal{D}(\alpha, c_\alpha)}\leq & P(\mathcal{G}_1 \cap \dots \cap \mathcal{G}_{J-1}) \big(P(\mathcal{G}_J) + c_\alpha r^{-\alpha} \big)\\[2mm]
      & \leq & \dots \leq \big(P(\mathcal{G}_1) + c_\alpha r^{-\alpha} \big) \dotsm \big(P(\mathcal{G}_J) + c_\alpha r^{-\alpha} \big).
    \end{array}
  \end{equation}
  Establishing the desired inequality.
\end{proof}

We now turn to the case of dependent percolation.

\begin{proof}[Proof of Theorem~\ref{t:main}]
  Recall our assumption on the decay exponent
  \begin{equation}
    \alpha > \Big(1 \vee \frac{d_i (d_u - 1)}{d_i - 1} \Big) d_u - d_l.
  \end{equation}
  We can clearly find $\gamma > 1 \vee \big(d_i (d_u - 1)/(d_i - 1)\big)$ and $d < d_i$ such that
  \begin{equation}
    \alpha > \gamma d_u - d.
  \end{equation}
  We now use this $\gamma$ in the definition of $L_k$ in \eqref{e:Lk}.

  Given such $\gamma$, pick any $\beta > \alpha \vee \gamma(1 + \chi)$ (recall Lemma~\ref{l:lego}) and choose an integer $J = J(\alpha, \gamma, \beta, d_u, d) \geq 2$ such that
  \begin{equation}
    \label{e:J_main}
    J \big(\alpha - (\gamma d_u - d) \big) > \gamma \beta.
  \end{equation}

  As before, we intend to establish a fast decay for
  \begin{equation}
    \label{e:p_k}
    p_k = \sup_{x \in V} \mathbb{P}[\mathcal{S}(x, L_k)],
  \end{equation}
  which will be done by induction.

  \newconst{c:s_logr2}
  Given $x \in V$, for $k \geq \useconst{c:s_logr2}(\gamma,d_l,d)$, we can use Lemma~\ref{l:Kset} (with $s = L_k$, $r = L_{k+1}$) to obtain a set $K \subseteq B(x, 5L_{k+1}/6)$ such that
  \begin{enumerate}[\quad a)]
  \item $|K| \leq \useconst{c:Kset} \frac{L_{k+1}^{d_u}}{L_k^d} = \useconst{c:Kset} L_k^{\gamma d_u - d}$ and
  \item $B(x, 4L_{k+1}/6) \subseteq B(K, L_k/6)$.
  \end{enumerate}

  \newconst{c:2points2}
  Our induction will rely on the cascading property of the separation events $\mathcal{S}(x, L_k)$.
  We now fix $\useconst{c:2points2} = \useconst{c:2points2}(d_i, c_i, d_u, c_u, d, \gamma, \alpha, \beta, J) > \useconst{c:s_logr2}$ such that for $k \geq \useconst{c:2points2}$ we have
  \begin{equation}
    \Big\lfloor \useconst{c:cascade} L_k^{\gamma \big( \tfrac{d_i - 1}{d_i}\big) - (d_u - 1)} \Big\rfloor > J.
  \end{equation}

  Again by Lemma~\ref{l:cascade}, we have
  \begin{equation}
    \mathbb{P}[\mathcal{S}(x, L_{k+1})] \leq \mathbb{P}\Big[
    \begin{array}{c}
      \text{there exist $y_1, \dots, y_{J} \in K$, within distance $3L_k$}\\
      \text{and such that $\mathcal{S}(y_i, L_k)$ hold for all $i = 1, \dots, J$}
    \end{array}
    \Big],
  \end{equation}
  and using Lemma~\ref{l:several_boxes} one obtains
  \begin{equation}
    \label{e:inductive_dependent}
    p_{k+1} \leq \big(\useconst{c:Kset} L_k^{\gamma d_u - d}\big)^J (p_k + c_\alpha L_k^{-\alpha})^J.
  \end{equation}
  Again, our aim is to use induction to show that
  \begin{equation}
    \label{e:pk_decay_2}
    p_k \leq L_k^{-\beta} \text{ for $k$ large enough}.
  \end{equation}

  \newconst{c:k_dependent}
  Suppose first that $p_k \leq L_k^{-\beta}$ and use \eqref{e:inductive_dependent} to estimate
  \begin{equation}
      \frac{p_{k+1}}{L_{k+1}^{-\beta}} \leq (2\useconst{c:Kset})^J(c_\alpha \vee 1)^J L_k^{J(\gamma d_u - d) + \gamma \beta - J(\beta \wedge \alpha)} \overset{\beta \geq \alpha}\leq (2\useconst{c:Kset})^J(c_\alpha \vee 1)^J L_k^{- J(\alpha - (\gamma d_u - d)) + \gamma \beta}.
  \end{equation}
  using \eqref{e:J_main}, we conclude that the above is smaller or equal to one for any $k \geq \useconst{c:k_dependent}$, where the constant $\useconst{c:k_dependent}$ is allowed to depend on $d_i, c_i, d_u, c_u, d, \alpha, c_\alpha, \gamma, \beta$ and $J$.

  This means that if \eqref{e:pk_decay_2} holds for a given $k_o \geq \useconst{c:k_dependent}$, then it must also hold for all $k \geq k_o$.
  It is important to notice that the dependence of $\useconst{c:k_dependent}$ on $\mathbb{P}$ is made explicit through the constants $\alpha$ and $c_\alpha$.
  This allows us to conclude that there exists a $p_* = p_*(d_i, c_i, d_u, c_u, \alpha, c_\alpha, J, \gamma, \beta) < 1$ such that, if $\inf_{x \in V} \mathbb{P}[\text{$x$ is open}] > p_*$, then \eqref{e:pk_decay_2} holds for $k_o$ and therefore for every $k \geq k_o$.

  To finish the proof one should simply recall that we have chosen $\beta > \gamma(1 + \chi)$ and employ the Lemma~\ref{l:lego} again.
  For the simple existence and uniqueness of the infinite open cluster, we can drop the dependence of $p_*$ on $J, \gamma$ and $\beta$ as in the statement of the theorem, since we can fix their values.
\end{proof}

\section{Open questions and remarks}
\label{s:questions}

In writing this paper, we have tried to balance between generality and simplicity.
In particular, we believe that there should be plenty of room for improvements in the presented results.
Below we point out some interesting directions and questions to pursue.

\begin{remark}
  \label{r:questions}
  $(a)$ The decay in \eqref{e:second_cluster} is not sharp, as we briefly commented below Theorem~\ref{t:bernoulli}.
  It would be interesting to obtain better tails for the size of finite percolation clusters and compare them with lower bounds obtained by explicit examples.

  \vspace{4mm}
  $(b)$ The relation between isoperimetric inequalities and heat kernel estimates for simple random walks on graphs is very well established and has resulted in extensive research, see for instance \cite{Sal97}, \cite{W00} and \cite{zbMATH05636419}.
  We believe that the same should be tried in the context of percolation, both improving the results of this article and providing examples of graphs satisfying local isoperimetric inequalities.

  \vspace{4mm}
  $(c)$ In Remark~\ref{r:isoperimetric} we have exhibited two examples of graphs satisfying \eqref{e:standard_iso} but not Definition~\ref{d:isoperimetric}.
  However we have not been able to find an example of a transitive, amenable graph that lies in between these two definitions.
  This raises the question of how \eqref{e:standard_iso} and Definition~\ref{d:isoperimetric} relate to each other when one imposes further conditions on the graph.
  Answering this question would of course have direct consequences to Question~\ref{q:BS} and to the classes of graphs covered by Corollary~\ref{c:ising}.

  \vspace{4mm}
  $(d)$ A trivial consequence of Theorem~\ref{t:bernoulli} is the following.
  If for some graph $G$ there is percolation at $p_c$, then no infinite critical cluster can contain a sub-graph satisfying the hypothesis of Theorem~\ref{t:bernoulli}.
  However, proving the existence of such a sub-graph, only assuming that there is percolation at criticality, seems to be a very difficult task.
\end{remark}

\renewcommand{\theequation}{A.\arabic{equation}}
\renewcommand{\thetheorem}{A.\arabic{theorem}}
\setcounter{theorem}{0}

\bibliographystyle{../BibTeX/jcamsalpha}
\bibliography{../BibTeX/all}

\end{document}